\documentclass[12pt,reqno]{amsart}

\usepackage{amsmath, amsthm, amssymb, esint}
\usepackage{amsfonts}
\usepackage{bbm}
\usepackage[ansinew]{inputenc}
\usepackage[dvips]{epsfig}
\usepackage{graphicx}
\usepackage[english]{babel}
\usepackage{enumerate}
\usepackage{enumitem}
\usepackage{hyperref}
\usepackage{fbox}
\usepackage{mathrsfs}
\usepackage{color}
 \usepackage{accents}

\newlist{steps}{enumerate}{1}
\setlist[steps, 1]{label = {\it \underline{\smash{Step~\arabic*:}}}, ref={\it Step~\arabic*}, wide, itemindent = 3pt}
 
\theoremstyle{plain}
\newtheorem{thm}{Theorem}[section]

\newtheorem{lem}[thm]{Lemma}
\newtheorem{prop}[thm]{Proposition}

\theoremstyle{definition}
\newtheorem{defi}[thm]{Definition}

\theoremstyle{remark}
\newtheorem{rem}[thm]{Remark}

\numberwithin{equation}{section}

\newcommand{\fls}{(-\Delta)^s}
\newcommand{\R}{\mathbb{R}}

\newcommand{\N}{\mathbb{N}}

\newcommand{\eps}{\varepsilon}

\newcommand{\be}{{\boldsymbol{e}}}

\def\fls{{(-\Delta)^s}}

\def\be{{\boldsymbol{e}}}

\def\R{\mathbb{R}}

\def\N{\mathbb{N}}

\def\L{\mathcal{L}}
\def\LL{\mathfrak{L}}
\def\LLL{\LL^\omega_s(\lambda, \Lambda)}
\def\I{\mathcal{I}}
\def\J{\mathcal{J}}
\def\II{\mathfrak{I}}
\def\III{\II^\omega_s(\lambda, \Lambda)}
\def\IIL{\mathfrak{I}^\omega_s(\lambda, \Lambda)}
\def\A{\mathcal{A}}
\def\B{\mathcal{B}}

\def\M{\mathcal{M}}

\newcommand{\Mp}{{\M^+}}
\newcommand{\Mm}{{\M^-}}
\newcommand{\Mpm}{{\M^\pm}}

\DeclareRobustCommand{\S}{\ifmmode\mathsection\else\textsection\fi}
\renewcommand\mathsection{\operatorname{\mathbb{S}}}

\DeclareMathOperator{\dist}{dist}

\newcommand{\average}{{\mathchoice {\kern1ex\vcenter{\hrule height.4pt
width 6pt depth0pt} \kern-9.7pt} {\kern1ex\vcenter{\hrule
height.4pt width 4.3pt depth0pt} \kern-7pt} {} {} }}

\def\R{\mathbb{R}}

\title[Smooth approximations for nonlocal elliptic equations]{Smooth approximations for fully nonlinear\\ nonlocal elliptic equations}

\author{Xavier Fern\'andez-Real}
\address{EPFL SB, Station 8, CH-1015 Lausanne, Switzerland}
\email{xavier.fernandez-real@epfl.ch}



\thanks{The author was supported by the Swiss National Science Foundation (SNF grants 200021\_182565 and PZ00P2\_208930),   by the Swiss State Secretariat for Education, Research and lnnovation (SERI) under contract number MB22.00034, and by the AEI project PID2021-125021NAI00 (Spain).}

\begin{document}

\begin{abstract}
We show that any viscosity solution to a general fully nonlinear nonlocal elliptic equation can be approximated by smooth ($C^\infty$) solutions. 
\end{abstract}

\maketitle

%%%%%%%%%%%%%%%%%%%%%%%%%%%%%%%%%%%%%%%%%%%%%%%%%%%%%%%%%%%%%%%%%%%%%%%%%%%%%%%%%%%%%%%%%%%%%%%%%%%%%%%%%%%%%%%%%%%%%%%%%%%%%%%%%%%%%%%%%%%%%%%%%%%%%%%%%%%%%%%%%%%%%%%%%%%%%%%%%%%%
\section{Introduction}

The regularization of solutions to elliptic equations is a fundamental technique to generalize a priori regularity estimates to full classes of solutions (see \cite[Section 7.2]{GT77}, \cite[Section 5.3]{Eva98}, or \cite[Chapter 2]{FR22}). For weak or distributional solutions of  linear and translation invariant equations this is accomplished, for example, by convolving the solution with a smooth mollifier.

For viscosity solutions of fully nonlinear elliptic PDE, 
\[
F(D^2 u) = 0\quad\text{in}\quad B_1
\]
 (which in general, have no strong solutions), such a regularization procedure has been done in less straight-forward ways (see \cite{CSou08, CSou10};  or \cite{CS10} where the authors use nonlocal techniques). In neither of these cases, however, the approximation is done by solutions within the same class as the limit\footnote{It is actually an open problem to decide whether any viscosity solution of a fully nonlinear elliptic PDE can be approximated by smooth solutions, \cite{CS10}.}. 

In the nonlocal case, nevertheless, there is a natural way to regularize solutions of fully nonlinear nonlocal equations (see \eqref{eq:Iofheform00}), 
\begin{equation}
\label{eq:sol}
\I(u, x) = 0\quad\text{in}\quad B_1,
\end{equation}
by substituting the corresponding kernels of the linear operators near the origin by that of the fractional Laplacian; see \cite{CS11b} (and also \cite{CS11}), where Caffarelli and Silvestre use (and prove) that solutions to translation invariant and concave nonlocal equations (with $s > 1/2$  and smooth kernels) can be approximated by strong solutions (i.e., $C^{2,\alpha}$) to equations in the same class (see also \cite{Kri13} for a similar procedure for general elliptic operators with $s  >1/2$).  In some settings (see \cite{RTW23}), however, it is sometimes necessary to approximate by even more regular solutions (say, $C^4$), and the results in \cite{CS11b}-\cite{CS11} do not apply. In fact, for nonlocal equations, the approximating procedure remains open in the following cases:
\begin{itemize}
\item Approximation by smooth solutions, more regular than $C^{2,\alpha}$. 
\item General fully nonlinear equations. 
\item Non-translation invariant equations, for all $s\in (0, 1)$ and in different regularity classes. 
\end{itemize}

In this work  we tackle all three of the previous situations, and provide smooth ($C^\infty$) approximations of general fully nonlinear equations. This is a phenomenon that has no local counter-part described in the literature.

Observe that, in general, solutions to fully nonlinear nonlocal equations, \eqref{eq:Iofheform00}, are not better than $C^{2s+1}$, and it was unclear whether smooth solutions exist (outside of trivial settings) or how common they are. The only results we know in this direction are due to Yu in \cite{Yu172, Yu17}, who proved that certain special classes of equations admit smooth solutions. We prove here that this is, in fact, a common phenomenon: any solution to \eqref{eq:sol} can be approximated by $C^\infty$ solutions of equations in the same class. 

We expect our results to be useful in various settings. Some examples can be seen in \cite{Kri13, Ser15, FR23} to prove regularity estimates for general nonlocal fully nonlinear equations; in \cite{CDV22, RTW23} to apply Bernstein's technique in a nonlocal setting and obtain semiconvexity estimates; and in this same manuscript, where in subsection~\ref{ssec:equivalence} we use the regularization to prove the equivalence between distributional and viscosity solutions for linear translation invariant equations, for any $s\in (0, 1)$. 

We consider the class of linear operators with kernels comparable to the one of fractional Laplacian. Namely, we consider operators of the form 
\begin{equation}
\label{eq:Lu1}
\begin{split}
\L_x u(x) & = {\rm P.V.} \int_{\R^n}\big(u(x)-u(x+y)\big)K(x, y) \, dy \\
& = \frac12 \int_{\R^n}\big(2u(x)-u(x+y)-u(x-y)\big)K(x, y)\,dy,
\end{split}
\end{equation}
with
\begin{equation}
\label{eq:Kint1}
K(x, y) = K(x, -y)\qquad\text{in}\quad \R^n\,
\end{equation}
and
\begin{equation}
\label{eq:compdef}
0< \frac{\lambda}{|y|^{n+2s}}\le K(x, y) \le \frac{\Lambda}{|y|^{n+2s}}\qquad\text{in}\quad   \R^n.
\end{equation}

By considering a (concave) modulus of continuity $\omega:\R^+\to\R^+$, we will also assume a weak form of continuity with respect to the $x$-variable of the kernels:
\begin{equation}
\label{eq:contwK}
\int_{B_{2r}\setminus B_r} \left|K(x, y) - K(x', y)\right|\, dy \le \frac{\omega(|x-x'|)}{r^{2s}},\quad\text{for any} \ x, x'\in \R^n, \ r > 0. 
\end{equation}

We thus consider the class of linear nonlocal operators given by:
\[
\LL^\omega_s(\lambda, \Lambda) :=\big\{\L_x : 
\text{\eqref{eq:Lu1}-\eqref{eq:Kint1}-\eqref{eq:compdef}-\eqref{eq:contwK} holds}
 \big\},
\]
and the class $\II^\omega_s(\lambda, \Lambda)$  of fully nonlinear integro-differential operators of the form:
\begin{equation}
\label{eq:Iofheform00}
\I (u, x) = \inf_{b\in \B}\sup_{a\in \A}\big\{-\L_{ab, x} u(x) + c_{ab}(x)\big\}\quad\text{with}\quad \L_{ab, x}\in \LL^\omega_s(\lambda, \Lambda)
\end{equation}
given by 
\[
\II^\omega_s(\lambda, \Lambda) :=\left\{\I : \begin{array}{l}
\I \text{ is of the form \eqref{eq:Iofheform00}}, ~~\I(0, x)\in L^\infty(\R^n),~\text{and}\\
(c_{ab}(x))_{ab}~\text{have a common modulus of continuity $\omega$}
\end{array} \right\}.
\]

Our main goal is to prove the following, where $w_s\in L^1(\R^n)$ is the weight $w_s(x) = \frac{1}{1+|x|^{n+2s}}$. We refer to Remarks~\ref{rem:cinftyoperators} to \ref{rem:last}  below for a further characterization of the objects involved in the statement. 
\begin{thm}
\label{thm:cinftysol00}
Let $s\in (0, 1)$, and let $\I \in \III$. Let $u\in C(B_1)\cap L^\infty(\R^n)$ be any viscosity solution of
\[
\I (u, x) = f(x)\quad\text{in}\quad B_1
\]
for some $f\in C(B_1)$ with modulus of continuity $\omega$. 

Then, there exist sequences  of functions $u^{(\eps)}, f_\eps\in C^\infty_c(\R^n)$ such that 
\[
\begin{split}
    u^{(\eps)}\to u&\quad\text{uniformly in $B_{3/4}$ and in $L^1(\R^n; w_s)$},\\
     f_\eps\to f&\quad\text{uniformly in $B_{3/4}$,}
  \end{split}
 \]
and a sequence of operators $\I_\eps\in \III$ with 
\[
\I_\eps (0, x) \to \I(0, x) \quad\text{uniformly in $B_{3/4}$},\]
as $\eps \downarrow 0$, such that
\[
\I_\eps(u_\eps, x) = f_\eps(x) \quad\text{in}\quad B_{3/4}.
\]
Moreover, 
\[
\|u^{(\eps)}\|_{L^\infty(\R^n)}\le C \left(\|u\|_{L^\infty(\R^n)} + \|\I(0, \cdot) - f\|_{L^\infty(B_{3/4})} +\omega(\eps)\right)
\]
for some $C$ depending only on $n$, $s$, $\lambda$, and $\Lambda$. Finally, if $\I$ is translation invariant (resp. concave), then $\I_\eps$ are translation invariant (resp. concave).  
\end{thm}

Some remarks are in order: 

\begin{rem}
\label{rem:cinftyoperators}
The new operators $\I_\eps$ are $C^\infty$, in the sense that for any $w\in C^\infty_c(\R^n)$ we have $\I_\eps(w, x) \in C^\infty(\R^n)$, with vanishing derivatives at infinity; see the proof of Proposition~\ref{prop:cinfty1}.  
\end{rem}

\begin{rem}
The right-hand side is $f_\eps = f*\varphi_\eps$, for some mollifier $\varphi_\eps$ (see \eqref{eq:mollifiervarphi}-\eqref{eq:mollifiervarphi2}. Moreover, if $u\in C^\alpha(\R^n)$, then there exists $\delta > 0$ depending only on $\alpha$, $n$, $s$, $\lambda$, and $\Lambda$, such that 
\[
\|u^{(\eps)}\|_{C^\delta(\R^n)}\le C \left(\|u\|_{C^\alpha(\R^n)}+ \|\I(0, \cdot) - f\|_{L^\infty(B_{3/4})} +\omega(\eps) \right),
\]
for some $C$ depending only on $n$, $s$, $\alpha$, $\lambda$, and $\Lambda$. 
\end{rem}

\begin{rem}
\label{rem:above}
If the operator $\I$ has higher regularity in $x$ and $y$, this is also inherited by $\I_\eps$. Namely,  if $\I\in \II^\mu_s(\lambda, \Lambda; \theta)$ for some $\theta>0$ and $\mu > 0$ (see the notation in subsection~\ref{ssec:notation} below), then $\I_\eps\in \II_s^\mu(\lambda, \Lambda; \theta)$ with $[\I_\eps]^x_{\mu}\le C[\I]^x_{\mu}$ and $[\I_\eps]^y_{\theta}\le C[\I]^y_{\theta}$, and $C$ depending only on $n$, $s$, $\lambda$, $\Lambda$, $\mu$, and $\theta$; see Remark~\ref{rem:reg_inherited}. The same conclusion also holds for pointwise norms; see Remarks~\ref{rem:pointwise} and \ref{rem:pointwise2}.
\end{rem}

\begin{rem}
The new operators have kernels that are convex combinations of~\eqref{eq:Kepsdef}, where $K$ denotes any kernel of $\I$. Thus, any property that is preserved by this operation is actually inherited by the new operators $\I_\eps$(for example, the linear operators in \cite[Chapter 2]{FR23}). 
\end{rem}

\begin{rem}
\label{rem:last} In Theorem~\ref{thm:cinftysol00}, one could take instead $u\in C(B_1)$ with controlled growth at infinity, $|u(x)|\le C(1+|x|^{2s-\tau}) $ in $\R^n$ for some $\tau > 0$.

If one wants to take $u\in C(B_1) \cap L^1(\R^n; w_s)$, however, it can be done provided that the continuity assumption in \eqref{eq:contwK} is understood in a pointwise sense, $|K(x, y) - K(x', y)|\le \omega(|x-x'|) |y|^{-2s-n}$. In the same way, the results in Remark~\ref{rem:above} would be true for pointwise estimates in $x$ and $y$ as well (like those in Remark~\ref{rem:pointwise}). 
\end{rem}

\subsection{Outline of the proof} After introducing some preliminary notation and results in Section~\ref{sec:prel_steps}, the proof is then divided into three parts. 

In the first part, Section~\ref{sec:aprox_reg_kernels}, we use the ideas from \cite{CS11b} (and \cite{Kri13, Ser15}) to obtain  a detailed version of \cite[Lemma 2.1]{CS11b} in this more general setting, Proposition~\ref{prop:regularization}, where we approximate solutions to fully nonlinear equations by solutions to more regular equations, and we believe is a result of independent interest. 

In Section~\ref{ssec:classicalapprox} we then use the ideas of the previous section to construct a sequence of strong solutions, globally H\"older and compactly supported, to general fully nonlinear translation invariant equations. %Notice that, in general, it is not true that viscosity solutions to nonlinear equations are strong solutions (unless one chooses the class of operators appropriately; see \cite{Yu17}).

Finally, in Section~\ref{sec:proofmain} we prove our main result, Theorem~\ref{thm:cinftysol00}, by regularizing the $\inf\sup$ structure of the fully nonlinear operators. In the local setting, $F(D^2u) = 0$, this would be accomplished by directly regularizing $F$. In the nonlocal setting, the analogue of $F$ would be defined on an infinite dimensional space (the space of kernels), and one needs to be more careful about such regularization.

\section{Preliminary steps}\label{sec:prel_steps}

In this section we introduce the notation used throughout the work, as well as some preliminary results that will be useful in the following proofs. 

\subsection{Notation} \label{ssec:notation} We will use kernels that have higher-order regularity in both $x$ and $y$. Let us introduce the corresponding spaces. 

Regarding the regularity in $x$, an analogous condition to \eqref{eq:contwK} can also be considered for any power. Thus, in general, for a given $\mu > 0$ with $\lceil \mu-1 \rceil = m$, we can impose
\begin{equation}
\label{eq:contwK_mu}
\begin{split}
\int_{B_{2r}\setminus B_r} \left|D^m_x K(x, y) - D^m_x K(x', y)\right|& \, dy \le C_\circ\frac{ |x-x'|^{\mu-m}}{r^{2s}},\\& 
\text{for any} \ x, x'\in \R^n, \ r > 0.
\end{split}
\end{equation}
%and if $\mu = m \in \N$,
%\begin{equation}
%\label{eq:contwK_mu__2}
%\begin{split}
%\int_{B_{2r}\setminus B_r} \left|D^m_x K(x, y)\right|& \, dy \le \frac{C_\circ}{r^{2s}}, \quad \text{for any} \ x \in \R^n, \ r > 0.
%\end{split}
%\end{equation} 
We can then define the corresponding semi-norm as the best possible constant $C_\circ$ satisfying such conditions for $\mu > 0$: 
\begin{equation}\label{Cmu-assumption_x}
[K]^x_\mu := \inf\left\{ C_ \circ > 0 : \eqref{eq:contwK_mu} \  \text{holds}. \right\}. 
\end{equation}

On the other hand, we can also impose a higher-order regularity in the $y$ variable. In this case, the analogous condition we will require is that, for a given $\mu > 0$ with $\lceil\mu-1\rceil = m$,
\begin{equation}\label{Cmu-assumption}
\begin{split}
\int_{B_{2r}\setminus B_r}\big|D^m_y K(x, z-y)-D^m_y K(x, z'-y)\big| & \, dy \le  C_\circ\frac{|z-z'|^{\mu-m}}{r^{2s+\mu}}\\
& \text{for all}\ r > 0, \ z, z'\in B_{r/2}, \ x\in \R^n.
\end{split}
\end{equation}
%and if $\mu = m\in \N$, 
%\begin{equation}\label{Cmu-assumption__2}
%\begin{split}
%\int_{B_{2r}\setminus B_r}\big|D^m_y K(x, z-y)\big| & \, dy \le  \frac{C_\circ}{r^{2s+m}} \ \text{for all}\ r > 0, \ z\in B_{r/2}, \ x\in \R^n.
%\end{split}
%\end{equation}
We denote 
\begin{equation}\label{Cmu-assumption_2}
[K]^y_\mu := \inf\left\{ C_ \circ > 0 : \eqref{Cmu-assumption}\  \text{holds}. \right\}. 
\end{equation}
If $\L_x$ is an operator of the form \eqref{eq:Lu1} with kernel $K = K(x, y)$, we denote
\begin{equation}
\label{eq:LKmu}
[\L_x]^y_\mu := [K]^y_\mu. 
\end{equation}

\begin{rem}
\label{rem:pointwise}
The previous conditions are satisfied when the kernels are H\"older continuous with an appropriate scaling. Namely, for condition \eqref{eq:contwK_mu} to hold it is enough to ask $[K]_{C_x^\mu(B_{2r}\setminus B_r)} \le Cr^{-2s-n}$ for all $r > 0$, and in the case of condition \eqref{Cmu-assumption} it is enough to have $[K]_{C_y^\mu(B_{2r}\setminus B_r)} \le Cr^{-2s-n-\mu}$ for all $r > 0$. These are the type of pointwise norms used, e.g., in \cite{CS11b, CS11}. 
\end{rem}

We then  define the classes  $\LL^\omega_s(\lambda, \Lambda;\mu)$ and $\LL^\mu_s(\lambda, \Lambda)$ for  $\mu  > 0$ as follows:
\begin{defi}
\label{defi:LL}
Let $s\in (0, 1)$, $0< \lambda \le \Lambda$, and let $\omega:\R^+\to\R^+$ be a modulus of continuity. We define, for $\mu > 0$, 
\[
\LL^\omega_s(\lambda, \Lambda; \mu) :=\big\{\L_x \in \LLL : [\L_x]^y_{\mu} < \infty\big\}.
\]
 We also denote $\LL^\omega_s(\lambda, \Lambda; 0) := \LL^\omega_s(\lambda, \Lambda) $.
 
On the other hand  we define, for a given $\mu > 0$, 
 \[
 \LL_s^\mu(\lambda, \Lambda) :=\big\{\L_x : 
\text{\eqref{eq:Lu1}-\eqref{eq:Kint1}-\eqref{eq:compdef} holds, and } [\L_x]^x_{\mu} < \infty\big\}
 \big\},
 \]
 and we denote $\LL_s^\infty(\lambda, \Lambda) = \bigcap_{\mu > 0}  \LL_s^\mu(\lambda, \Lambda)$. 
\end{defi}

%
%Notice that the class $\LLL$ is \emph{scale invariant}. 
%That is, for any $\L \in \LLL$ there is $\L_r\in \LLL$ such that
%\begin{equation}
%\label{eq:scaleinvariance_comp}\index{Scale invariance!Operators comparable to fractional Laplacian}
%\big(\L_r u(r\, \cdot\,)\big)(x) = r^{2s}(\L u)(rx)
%\end{equation}
%for every $u$ such that $\L u(rx)$ is well-defined. 
%More precisely, if $\L$ has kernel~$K$, then the kernel $K_r$ of $\L_r$ is given by $K_r(y) := r^{n+2s} K(ry)$. 

Of course, we also have the corresponding regular classes of fully nonlinear operators of the form \eqref{eq:Iofheform00}, $\II^\omega_s(\lambda, \Lambda; \mu)$ and  $\II^\mu_s(\lambda, \Lambda)$:

\begin{defi}
\label{defi:II}
Let $s\in (0, 1)$, $0< \lambda \le \Lambda$, and let $\omega:\R^+\to\R^+$ be a modulus of continuity. We define, for $\mu > 0$, 
\[
\II^\omega_s(\lambda, \Lambda; \mu) :=\big\{\I \in \IIL: [\I]^y_{\mu}<+\infty 
 \big\}.
\]
where for $\I \in \III$ we denote
\[
[\I]^y_{\mu} := \sup_{(a,b)\in \A\times\B} [\L_{ab, x}]^y_{\mu}.
\]
In particular, in the expression \eqref{eq:Iofheform00} we have that $\L_{ab, x}\in \LL^\omega_s(\lambda, \Lambda; \mu)$ for all $(a, b)\in \A\times \B$. When $\mu = 0$, we denote furthermore $\II^\omega_s(\lambda, \Lambda; 0) := \III$.

We also define, given $\mu > 0$, 
\[
\II^\mu_s(\lambda, \Lambda; \mu) :=\big\{\I \in \IIL: [\I]^x_{\mu}<+\infty 
 \big\}.
\]
where for $\I \in \III$ we denote
\[
[\I]^x_{\mu} := \sup_{(a,b)\in \A\times\B} [\L_{ab, x}]^x_{\mu}, 
\]
and    $\II_s^\infty(\lambda, \Lambda) = \bigcap_{\mu > 0}  \II_s^\mu(\lambda, \Lambda)$. 
\end{defi}

  The  extremal operators corresponding to the class $\LLL$ have a relatively simple closed expression: 
\begin{equation}
\label{eq:Mpexplicit}
\begin{split}
\mathcal{M}_{s, \lambda, \Lambda}^+u(x) = &\frac12  \int_{\R^n}\bigg\{  \Lambda\big(u(x+y)+u(x-y)-2u(x)\big)_+ \\
&\qquad\quad  - \lambda\big(u(x+y)+u(x-y)-2u(x)\big)_-\bigg\}\frac{dy}{|y|^{n+2s}},
\end{split}
\end{equation}
and 
\begin{equation}
\label{eq:Mpexplicit2}
\begin{split}
\mathcal{M}_{s, \lambda, \Lambda}^-u(x) = &\frac12  \int_{\R^n}\bigg\{  \lambda\big(u(x+y)+u(x-y)-2u(x)\big)_+ \\
& \qquad - \Lambda\big(u(x+y)+u(x-y)-2u(x)\big)_-\bigg\}\frac{dy}{|y|^{n+2s}}.
\end{split}
\end{equation}
Throughout this paper we will denote 
\begin{equation}
\label{eq:MMpm}
\Mp := \mathcal{M}_{s, \lambda, \Lambda}^+ \qquad\text{and}\qquad \Mm := \mathcal{M}_{s, \lambda, \Lambda}^-.
\end{equation}

The class $\II_s^\omega(\lambda, \Lambda)$ is uniformly elliptic with respect to $\LL_s^\omega(\lambda, \Lambda)$, that is, 
\[
\Mm v(x) \le \I(u+v, x) - \I(u, x) \le \Mp v(x). 
\]

\subsection{Preliminary results} The following are well-known results that we re-write here for the convenience of the reader. The first is a result regarding the regularity of $\L_x u$ when $u$ is regular.  

\begin{lem}
\label{lem:Lu_2}
Let $s\in (0, 1)$ and let $\mu$ with $\mu\not\in\mathbb N$. Then, if $\L_x\in \LL^\mu_s(\lambda, \Lambda)$,  for any $u\in C^{2s+\mu}(B_1)\cap C^\mu(\R^n)$ we have $\L_x u\in C_{\rm loc}^\mu(B_{1})$ and
\[
\|\L_x u\|_{C^\mu (B_{1/2})} \le C \Lambda \left(\|u\|_{C^{2s+\mu}(B_1)} +\|u\|_{C^{\mu}(\R^n)}\right),
\]
with $C$ depending only on $n$, $s$,  $[\L_x]_{\mu}^x$, and $\mu$.
\end{lem}

\begin{proof} The proof of this result is standard. We briefly sketch the main steps for completeness, and refer to \cite{FR23} for more details. 

Let us assume first that $\L_x = \L$ is translation invariant, with kernel $K(y)$ satisfying $|K(y)|\le \Lambda |y|^{-n-2s}$ in $\R^n$.  We fix a cut-off function $\eta\in C^\infty_c(\R^n)$ such that $\eta \ge 0$, $\eta \equiv 0$ in $\R^n\setminus B_{3/4}$ and $\eta \equiv 1$ in $B_{2/3}$, and define 
\[
u_1 := \eta u\qquad\text{and}\qquad u_2 := (1-\eta) u,
\]
and we bound separately the regularity of $\L u_1$ and $\L u_2$. 

To bound   $\|\L u_1\|_{C^\mu(B_{1/2})}$ we use $u_1\in C^{2s+\mu}(\R^n)$ with $\|u_1\|_{C^{2s+\mu}(\R^n)}\le C$. Then, if $ \mu< 1$, a  direct computation shows that
\[
\begin{split}
|u_1(x+y)+u_1(x-y)-2u_1(x)| &\le  C |y|^{2s+\mu},\\
|u_1(x)+u_1(-x)-2u_1(0)| &\le  C |x|^{2s+\mu},\\
|u_1(x\pm y)+u_1(-x\pm y)-2u_1(\pm y)| &\le  C |x|^{2s+\mu}.
\end{split}
\]
for $2s+\mu \le 2$, and 
\[
\begin{split}
& \big|u_1(x+y)+u_1(x-y)-2u_1(x)-u_1(y)-u_1(-y)+2u_1(0)\big| \leq  |y|^2 |x|^{2s+\mu-2},\\
& \big|u_1(x+y)+u_1(-x+y)-2u_1(y)-u_1(x)-u_1(-x)+2u_1(0)\big| \leq  |x|^2 |y|^{2s+\mu-2}
\end{split}
\]
if  $2<2s+\mu<3$. Together with the bounds on the kernel, this directly yields
\begin{equation}
\label{eq:weobtainagain}
\big|\L u_1(x) + \L u_1(-x) - 2 \L u_1(0) \big| \le C \Lambda |x|^\mu, 
\end{equation}
which is what we wanted (repeating around every point in $B_{1/2}$). On the other hand, if $\mu = k +\beta$ with $k \in \N$, $\beta\in (0, 1)$, we take $k$ derivatives of $\L u_1$ and repeat the arguments above, to obtain $[D^k \L u_1]_{C^\beta(B_{1/2})} \le C$. 

For the bound on $\|\L u_2\|_{C^\mu(B_{1/2})}$ (with  $\mu = k+\beta$ as above), we have 
\[
\begin{split}
\left|D^k \L u_2(x) - D^k \L u_2(0) \right|& = \left|\L D^k u_2(x) - \L D^k u_2(0) \right| \le C |x|^{\beta},
 \end{split}
\]
using that $D^k u_2\in C^\beta$  and the fact that $u_2 \equiv 0$ in $B_{2/3}$. 

Finally, to do the general non-translation invariant case, given a fixed point $x_i$ we denote $\L_{x_i}$ the translation invariant operator with kernel $K(x_i, y)$. Thus, if $\mu < 1$ we can write for $x_1, x_2\in B_{1/2}$, 
\[
|\L_{x_1} u(x_1) - \L_{x_2} u(x_2)|\le |\L_{x_1} u(x_1) - \L_{x_1} u(x_2)| + |\L_{x_1} u(x_2) - \L_{x_2} u(x_2)|,
\]
where the first term can be bounded as before, and the second term is bounded thanks to \eqref{eq:contwK_mu} and the fact that $u$ is $C^{2s+\mu}$ around $x_2$. If $\mu > 1$, we use the same reasoning, by taking the derivatives of $\L_x u(x)$ and using the bound \eqref{eq:contwK_mu} again.
\end{proof}

The following, is a direct consequence of the comparison principle for the extremal operators:  

\begin{lem}
\label{lem:Linftybound_visc} 
Let $s\in (0,1)$ and   $\Mpm $ be given by \eqref{eq:MMpm}. Let $u\in C(B_1)\cap L^\infty(\R^n)$ be a viscosity solution of 
\begin{equation}
\label{eq:sol_u_g}
\left\{
\begin{array}{rcll}
\Mp u & \ge & -C_\circ & \quad\text{in}\quad B_1,\\
\Mm u & \le & C_\circ & \quad\text{in}\quad B_1.
\end{array}
\right.
\end{equation}
Then, 
\[
\|u\|_{L^\infty(B_1)} \le \|u\|_{L^\infty(\R^n\setminus B_1)} + C C_\circ,
\]
for some constant $C$ depending only on $n$, $s$, $\lambda$, and $\Lambda$. 
\end{lem}
\begin{proof}
The proof is standard, by applying the maximum principle to the functions $u \pm C C_\circ \chi_{B_2}$. 
\end{proof}

The next result is on the interior regularity of solutions to non-divergence-form equations with bounded measurable coefficients:

\begin{thm}[\cite{Sil06, CS09}]\label{C^alpha-bmc}
Let $s\in(0,1)$ and let $\Mpm $ be given by \eqref{eq:MMpm}. Let $u\in C(B_1)\cap L^\infty(\R^n)$   be any viscosity solution to  a non-divergence-form equation with bounded measurable coefficients, i.e.,
\[\left\{
\begin{array}{rcll}
\Mp u &\geq & -C_\circ &\quad \textrm{in}\quad B_1\\
\Mm u &\leq & C_\circ &\quad \textrm{in}\quad B_1,
\end{array}
\right.\]
for some~$C_\circ\geq0$.
Then $u\in C^\gamma_{\rm loc}(B_1)$ with 
\[\|u\|_{C^{\gamma}(B_{1/2})}\leq C\left(\|u\|_{L^\infty(\R^n)}+C_\circ\right),\]
where $C$ and $\gamma>0$ depend only on $n$, $s$,  $\lambda$, and $\Lambda$.
\end{thm}

And the corresponding result regarding the regularity up to the boundary:

\begin{lem}[\cite{CS11}]
\label{lem:bdry_reg_int}
Let $s\in (0, 1)$, and let $g\in L^\infty(\R^n)\cap C^\alpha(B_2)$ for some $\alpha > 0$. Let $\Mpm$ be given by \eqref{eq:MMpm}, and let $u\in C(\Omega)\cap L^\infty(\R^n)$ be a viscosity solution of 
\[
\left\{
\begin{array}{rcll}
\Mp u & \ge & -C_\circ & \quad\text{in}\quad B_1,\\
\Mm u & \le & C_\circ & \quad\text{in}\quad B_1,\\
u & = & g & \quad\text{in}\quad \R^n\setminus B_1.
\end{array}
\right.
\]
Then, $u\in C^\delta(\overline{B_1})$ for some $\delta >0$ depending only on $\alpha$, $n$, $s$, $\lambda$, and $\Lambda$. 
\end{lem}

We will also make use of the notion of weak convergence of nonlocal operators, in particular to study limits of viscosity solutions; see \cite[Lemma 4.3]{CS11}.

\begin{defi}[Weak convergence of operators]
\label{defi:conv_I}
Let $s\in (0, 1)$. Let $(\I_k)_{k\in \N}$ be a sequence of operators with $\I_k\in \II^\omega_s(\lambda, \Lambda)$ and let $\I\in \II^\omega_s(\lambda, \Lambda)$. We say that $\I_k$ weakly converges to $\I$ in $\Omega$, and we denote it 
\[
\I_k\rightharpoonup \I\quad\text{in}\quad \R^n,
\] if for every $x_\circ\in \Omega$ and every function $v\in L^\infty(\R^n)$ such that $v$ is $C^2$ in  $B_r(x_\circ) \subset \Omega$, we have $\I_k(v, x) \to \I(v, x)$ uniformly in $B_{r/2}(x_\circ)$. 
\end{defi}

Finally, we also recall the following classical result on the interior regularity of solutions to equations with the fractional Laplacian (see, for example, \cite{RS16}): 
\begin{prop}[Interior estimates for viscosity solutions of $\fls$]
\label{prop:viscosity_fls}
Let $s\in (0, 1)$, and let $f\in C^\theta(B_1)$ for some $\theta\in [0, 1)$. Let $u\in C(B_1) \cap L^\infty(\R^n)$ satisfy
\[
\fls u = f \quad \text{in}\quad B_1
\]
in the viscosity sense. Then, if $2s+\theta\notin \N$,  $u\in C_{\rm loc}^{2s+\theta}(B_1)$ with 
\[
\|u\|_{C^{2s+\theta}(B_{1/2})}\le C \left(\|u\|_{L^\infty(\R^n)} + \|f\|_{C^\theta(B_1)}\right),
\]
for some $C$ depending only on $n$, $s$, and $\theta$. If $\theta = 0$ and $s = \frac12$, then $u\in C^{1-\delta}(B_1)$ for any $\delta > 0$. 
\end{prop}

\section{Approximation of equations with regular kernels}
\label{sec:aprox_reg_kernels}

  We start with a first approximation result in the case of regular kernels. We consider operators of the form 
\begin{equation}
\label{eq:Iepsisoftheform0}
\I (u, x) = \inf_{b\in \B}\sup_{a\in \A}\big\{-\L_{ab, x} u(x) + c_{ab}(x)\big\},\qquad \L_{ab, x}\in \LL_s(\lambda, \Lambda; \theta),
\end{equation}
from which we define its regularized version as
\begin{equation}
\label{eq:Iepsisoftheform}
\I_\eps (u, x) = \inf_{b\in \B}\sup_{a\in \A}\left\{-\L_{ab, x}^{(\eps)} u(x) + c_{ab}(x)\right\},\qquad \L_{ab, x}^{(\eps)}\in \LL_s(\lambda, \Lambda; \theta),
\end{equation}
where   the $c_{ab}$ are the same as above. The first regularization or approximation result  is then the following:

\begin{prop}
\label{prop:regularization}
Let $s\in (0, 1)$, and let $\I \in \II^\omega_s(\lambda, \Lambda; \theta)$ with $\theta\in [0, 1)$ be of the form \eqref{eq:Iepsisoftheform0}. Let us assume, moreover, that 
\[
\sup_{(a, b)\in \A\times\B} [c_{ab}]_{C^\theta(\R^n)}<\infty,
\]
where we denote $[\, \cdot\,]_{C^0} = {\rm osc}(\, \cdot\, )$. 

Let $u\in C(B_1)\cap L^\infty(\R^n)$ be any viscosity solution of
\[
\I (u, x) = 0\quad\text{in}\quad B_1. 
\]
Then, there exist  a sequence of functions,
\[\qquad \qquad u^{(\eps)}\in C^{2s+\theta}_{\rm loc}(B_{3/4})\cap {C(B_1)}\cap L^\infty(\R^n)\qquad \textrm{if}\quad 2s+\theta\notin\N,\]
or $u^{(\eps)}\in C_{\rm loc}^{1-\delta}(B_{3/4})\cap {C(B_1)}\cap L^\infty(\R^n)$ for any $\delta > 0$ if $\theta = 0$ and $s = \frac12$; and  a sequence of operators $\I_\eps\in \II^\omega_s(\lambda, \Lambda; \theta)$ of the form \eqref{eq:Iepsisoftheform} and satisfying $[\I_\eps]^y_{\theta}\le C_1 [\I]^y_{\theta}$ if  $\theta > 0$, with $C_1$ depending only on $n$, $s$, $\lambda$, $\Lambda$, and $\theta$,
such that, 
\[
\left\{
\begin{array}{rcll}
\I_\eps(u^{(\eps)}, x) & = & 0 & \quad\text{in}\quad B_{3/4}\\
u^{(\eps)} & = & u & \quad\text{in}\quad \R^n\setminus B_{3/4},
\end{array}
\right.
\]
and
\[
u^{(\eps)}\to u\quad\text{locally uniformly in $B_{3/4}$}.
\]
 Moreover, we have 
\begin{equation}
\label{eq:bound_reg}
\|u^{(\eps)}\|_{L^\infty(B_{3/4})}\le C \left(\| u\|_{L^\infty (\R^n)} + \|\I(0, x)\|_{L^\infty(B_{3/4})}\right)
\end{equation}
for some $C$ depending only on $n$, $s$, $\lambda$, and $\Lambda$. 
\end{prop}

Let us start with the construction of $\I_\eps$.  Let $\psi:[0, \infty) \to [0, \infty)$ be a given fixed cut-off function such that 
\[
\left\{\begin{array}{l}
\psi\in C^\infty_c([0, \infty)), \\
\psi = 1 \quad\text{in}\quad [0, 1/2],\\ 
\psi = 0 \quad\text{in}\quad [1, \infty),\\
\text{$\psi$ is monotone nonincreasing}.
\end{array}
\right.
\]
We also fix a mollifier $\varphi$ such that
\begin{equation}
\label{eq:mollifiervarphi}
\text{$\varphi\in C^\infty_c(B_1)
$ is radial, with $\varphi \ge 0$ and $\textstyle{\int_{B_1}}\varphi = 1$,}
\end{equation}
and we consider the rescalings
\begin{equation}
\label{eq:mollifiervarphi2}
\varphi_\eps(x) := \frac{1}{\eps^n}\varphi\left(\frac{x}{\eps}\right)\in C^\infty_c(B_\eps). 
\end{equation}

 Given $\L_x\in \LLL$ with kernel $K$ (which satisfies \eqref{eq:compdef}-\eqref{eq:contwK}) and $\eps > 0$, we define $\L^{(\eps)}_x$ to be the operator that has kernel $K_\eps$ given by 
 \begin{equation}
 \label{eq:Kepsdef}
 K_\eps(x, y) = \big(1-\psi(|y|/\eps) \big) (K(\cdot, y) * \varphi_\eps)(x)+ \psi(|y|/\eps)|y|^{-n-2s}.
 \end{equation}
 Notice that with this definition we still have $\L_x^{(\eps)}\in \LLL$. Moreover, we have:
 
% \begin{figure}
%\centering
%\makebox[\textwidth][c]{\includegraphics[scale = 1]{./Figures/fig.12_03_smaller.pdf}}
%\caption{\label{fig:12} The kernels $K$ and $K_\eps$ given by \eqref{eq:Kepsdef}.}
%\end{figure}
 
 \begin{lem}
 \label{lem:Lepsalpha} Let $s\in (0, 1)$. If $\L_x\in \LL^\omega_s(\lambda, \Lambda; \theta)$ for some $\theta>0$, then $\L_x^{(\eps)}\in \LL^\omega_s(\lambda, \Lambda; \theta)$ as well, with 
\[
 \big[\L_x^{(\eps)}\big]^y_{\theta} \le C    [\L_x]^y_{\theta}  
\]
 for some constant $C$ depending only on $n$, $s$, $\theta$, $\lambda$, and $\Lambda$. 
 \end{lem}
 \begin{proof}
 Let us define, for a kernel $J(x, y)$ with $\lceil \theta-1 \rceil = m$, 
 \[
 [J]_{\theta; r} := \sup_{x\in \R^n}\sup_{z, z'\in B_{r/2}} \int_{B_{2r}\setminus B_r} \frac{\left|D^m J(x, z-y) - D^m J(x, z'-y)\right|}{|z-z'|^\theta}\, dy,
 \]
(cf.   \eqref{Cmu-assumption}).

% and if $\theta = m\in \N$, 
% \[
% [J]_{\theta; r} := \sup_{x\in \R^n}\sup_{z \in B_{r/2}} \int_{B_{2r}\setminus B_r} \frac{\left|D^m J(x, z-y)\right|}{|z-z'|^m}\, dy,
% \]
%(cf. \eqref{Cmu-assumption__2}), so that 
% \[
% [J]^y_{\theta} = \sup_{r > 0} \, r^{2s+\theta}[J]_{\theta; r}.
% \]
 By the triangle inequality, we have 
 \[
 \begin{split}
 [K_\eps]_{\theta; r} & \le  [K *_x \varphi_\eps]_{\theta;r} + C (\Lambda+1)r^{-n-2s}  \big[\psi(|y|/\eps)\big]_{\theta;r} + \big[|y|^{-n-2s}\big]_{\theta;r}\\
& \le {C}\left([K]^y_{\theta}+  r^{-n+\theta}\eps^{n-\theta}[\psi]_{\theta;r/\eps} + \big[|y|^{-n-2s}\big]_{\theta;1}\right){r^{-2s-\theta}},
 \end{split}
 \]
 which, since $\psi$ is fixed and $[\psi]_{\theta;r/\eps} = 0$ for $r>100\eps$ or $r<\eps/100$, directly implies 
 \[
  [K_\eps]_{\theta; r} \le C([K]^y_{\theta}+1)r^{-2s-\theta},
 \]
and hence
 \[
 [K_\eps]^y_\theta \le C([K]^y_\theta+1).
 \]
 Since $[\L_x]^y_{\theta}\ge c > 0$ for all $\L_x \in \LL^\omega_s(\lambda, \Lambda; \theta)$, the result follows. 
 \end{proof}

 \begin{rem} \label{rem:reg_in_x} Notice that, in fact, the new operators $\L_x^{(\eps)}$ are regularizing in $x$, so  $\L_x^{(\eps)}\in \LL^{\omega}_s(\lambda, \Lambda)\cap \LL^{\infty}_s(\lambda, \Lambda)$ with bounds independent of $(a, b)\in \A\times \B$.  
 \end{rem}
 \begin{rem}
 \label{rem:pointwise2}
 The same proof would also yield that the operators   preserve pointwise norms, like the ones in Remark~\ref{rem:pointwise}.
 \end{rem}

Let now $\I \in \II^\omega_s(\lambda, \Lambda; \theta)$ for some $\theta>0$, i.e., of the form \eqref{eq:Iepsisoftheform0}. We define $\I_\eps$ as \eqref{eq:Iepsisoftheform} with $\L_{ab, x}^{(\eps)}$ given by \eqref{eq:Kepsdef}. By Lemma~\ref{lem:Lepsalpha} we immediately have that 
\[
\I_\eps\in \II^\omega_s(\lambda, \Lambda; \theta)
\]
 as well, with $[\I_\eps]^y_{\theta} \le C [\I]^y_{\theta}$  if $\theta >0$. Furthermore, $\I_\eps$ weakly converges to $\I$ as $\eps\downarrow 0$:

 \begin{lem} 
 \label{lem:weakconv}
Let $s\in (0, 1)$, and let $\I, \I_\eps\in \II^\omega_s(\lambda, \Lambda)$ be as above. Then 
 \[
 \I_\eps\rightharpoonup \I\quad\text{in}\quad \R^n,\quad \text{as}\quad \eps\downarrow 0,
 \]
 in the sense of Definition~\ref{defi:conv_I}.
 \end{lem}
 \begin{proof}
 Let $x_\circ\in \R^n$, and let  $v\in L^\infty (\R^n)$    such that it is $C^2$ in $B_r(x_\circ)$ for some $r > 0$. Let us compute, for any $x\in B_{r/2}(x_\circ)$ and $\L_x\in \II^\omega_s(\lambda, \Lambda)$ with kernel $K$,
 \[
 \L^{(\eps)} v(x) - \L v(x) =\frac12\int_{\R^n} \big(2v(x)-v(x+y)-v(x-y)\big)\big(K_\eps(x, y)-K(x, y)\big)\,dy.
 \]
 Since 
 \[
 K_\eps(x, y)-K(x, y) = \psi(|y|/\eps)\left(|y|^{-n-2s}-K(x, y)\right) + (K(\cdot, y) * \varphi_\eps)(x) - K(x, y)\]
  we can bound the right-hand side by 
 \[
\big| \L^{(\eps)} v(x) - \L v(x)\big|\le I  + II
 \]
 where, directly using that $|K(x, y)|\le \Lambda|y|^{-n-2s}$, we have
 \[
 I := C \int_{B_\eps} \big|2v(x)-v(x+y)-v(x-y)\big||y|^{-n-2s}\,dy,
 \]
 and 
 \[
 II := \int_{B_\eps} \int_{\R^n} \big|2v(x)-v(x+y)-v(x-y)\big|\big|K(x-z, y)-K(x, y)\big|\,dy \,\varphi_\eps(z)\, dz.
 \]
 Now, since $v$ is $C^2$ in $B_r(x_\circ)$ and taking $\eps < r/4$ we can bound $I$ by 
  \[
  \begin{split}
I& \le C \|v\|_{C^2(B_{3r/4}(x_\circ))}\int_{B_\eps} |y|^{-n-2s+2}\,dy \le C \|v\|_{C^2(B_{3r/4}(x_\circ))}\eps^{2-2s}.
\end{split}
 \]
 On the other hand, thanks to \eqref{eq:contwK}   we also have
 \[
 \begin{split}
 \int_{\R^n} \big|2v(x)-v(x+y)-v(x-y)&\big|\big|K(x-z, y)-K(x, y)\big|\,dy\\
 & \le C \left(\|v\|_{C^2(B_{3r/4}(x_\circ))} + \|v\|_{L^\infty(\R^n)}\right) \omega(|z|), 
 \end{split}
 \]
 and therefore
 \[
 II \le C \left(\|v\|_{C^2(B_{3r/4}(x_\circ))} + \|v\|_{L^\infty (\R^n)}\right) \omega(\eps). 
 \]

 Hence, we can bound 
\[
\begin{split}
\I_\eps (v, x) & = \inf_{b\in \B}\sup_{a\in \A}\left\{-\L_{ab, x}^{(\eps)} u(x) + c_{ab}(x)\right\}\\
& \le  \I(v, x) + C \left(\|v\|_{C^2(B_{3r/4}(x_\circ))} + \|v\|_{L^\infty (\R^n)}\right)\left(\eps^{2-2s}+ \omega(\eps)\right)
\end{split}
\] 
On the other hand, we also get similarly, 
\[
\I_\eps (v, x) \ge \I(v, x) - C \left(\|v\|_{C^2(B_{3r/4}(x_\circ))} + \|v\|_{L^\infty (\R^n)}\right)\left(\eps^{2-2s}+ \omega(\eps)\right).
\]
In all, we have proved that 
\[
\|\I_\eps(v, \cdot) - \I(v, \cdot)\|_{L^\infty(B_{r/2}(x_\circ))}\downarrow 0
\]
as $\eps\downarrow 0$, that is, $\I_\eps\rightharpoonup\I$.  
\end{proof}

We want to use the previous operators $\I_\eps$ to construct a series of regular solutions approximating a given solution. That is, let $\I \in \II^\omega_s(\lambda, \Lambda; \theta)$ for some $\theta\in [0, 1)$, and let $u\in C(B_1)\cap L^\infty(\R^n)$ be such that 
\begin{equation}
\label{eq:estrelletax2}
\I (u, x) = 0\quad\text{in}\quad B_1. 
\end{equation}
We then define $u^{(\eps)}$ to be the unique solution (see, for example, \cite{Ser15, Mou19})
\begin{equation}
\label{eq:uepsqualdef}
\left\{
\begin{array}{rcll}
\I_\eps(u^{(\eps)}, x) & = & 0 & \quad \text{in}\quad B_{3/4}\\
u^{(\eps)} & = & u & \quad \text{in}\quad \R^n\setminus B_{3/4}. 
\end{array}
\right.
\end{equation}

\begin{lem}
\label{lem:ueps1}
Let $s\in (0, 1)$ and $\I \in \III$. 
Let $u\in C(B_1)\cap L^\infty(\R^n)$ be any viscosity solution of \eqref{eq:estrelletax2}, and let $u^{(\eps)}\in C(B_{3/4})\cap L^\infty(\R^n)$ be the unique solution of \eqref{eq:uepsqualdef}.  Let  $\gamma > 0$ be given by Theorem~\ref{C^alpha-bmc}. Then 
\[
\| u^{(\eps)}\|_{L^\infty(B_{3/4})}+\| u^{(\eps)}\|_{C^\gamma(B_{1/2})} \le C \left(\| u\|_{L^\infty(\R^n)} + \|\I(0, x)\|_{L^\infty(B_{3/4})}\right),
\]
 for some $C$ depending only on $n$, $s$, $\lambda$, and $\Lambda$.
\end{lem} 
\begin{proof}
 Observe that,  
\[
\Mp v \ge \I_\eps(v, x) - \I_\eps(0, x) \ge \Mm v ,
\]
and since $\I_\eps(0, x) = \I(0, x)$, the bound on $\|u^{(\eps)}\|_{L^\infty(B_{3/4})}$ directly follows from the comparison principle in Lemma~\ref{lem:Linftybound_visc}, while the bound on $[u^{(\eps)}]_{C^\gamma(B_{1/2})}$ is a consequence of  Theorem~\ref{C^alpha-bmc} and the bound on $\|u^{(\eps)}\|_{L^\infty(B_{3/4})}$.
\end{proof}

We now want to show that the solutions $u^{(\eps)}$ are qualitatively regular (that is, strong solutions) in the interior of $B_{3/4}$. In order to do it, we use the structure of the operator $\I_\eps$, which  behaves like a fractional Laplacian. Thus, we need the interior estimates for viscosity solutions of equations with the fractional Laplacian, Proposition~\ref{prop:viscosity_fls}. The following is the qualitative result on the regularity of $u^{(\eps)}$:

\begin{lem}
\label{lem:regularization}
Let $s\in (0, 1)$. Let $u^{(\eps)}$ be defined as above, \eqref{eq:uepsqualdef}, for a fixed $\I \in \II^\omega_s(\lambda, \Lambda; \theta)$ with $\theta\in [0, 1)$ of the form \eqref{eq:Iepsisoftheform0}. Let us assume, moreover, that 
\[
\sup_{(a, b)\in \A\times\B} [c_{ab}]_{C^\theta(\R^n)}\le C_\circ <\infty,
\]
where we denote $[\, \cdot\,]_{C^0} = {\rm osc}(\, \cdot\, )$.  Then, if $2s+\theta\notin\N$, $u^{(\eps)}\in C^{2s+\theta}_{\rm loc}(B_{3/4})$. If $\theta = 0$ and $s  = \frac12$, we have $u^{(\eps)}\in C_{\rm loc}^{1-\delta}(B_{3/4})$ for any $\delta > 0$.
\end{lem}
\begin{proof}
For the sake of readability, let us denote $v = u^{(\eps)}$. Notice that, by Lemma~\ref{lem:ueps1} and a covering argument (or directly by Theorem~\ref{C^alpha-bmc}), we already know that $v$ is $C^\gamma$ inside~$B_{3/4}$. 

 We express now the operator $\I_\eps$ as follows:
\begin{equation}
\label{eq:Iepssplit}
\begin{split}
\I_\eps(v, x) & = -c^{-1}_{n,s}\fls v(x) +\inf_{b\in \B}\sup_{a\in \A}\left\{\tilde \L^{(\eps)}_{ab, x} v (x) + c_{ab}(x)\right\}\\
& =  -c^{-1}_{n,s}\fls v(x) + f_\eps(x),
\end{split}
\end{equation}
where we have denoted,
\[
\begin{split}
\tilde \L^{(\eps)}_{ab, x} v (x) & = \left(c^{-1}_{n,s}\fls-\L_{ab, x}^{(\eps)}\right) v(x)\\
&  = \frac12 \int_{\R^n}\big(2v(x) - v(x+y)-v(x-y)\big) \tilde K_\eps(x, y)\, dy\\
& = \int_{B_{\eps/2}^c}\big(v(x) - v(x+y)\big) \tilde K_\eps(x, y)\, dy,
\end{split}
\]
with
\[
\tilde K_\eps(x, y) = \big(1-\psi(|y|/\eps)\big) \left(|y|^{-n-2s}-(K_{ab}(\cdot, y)*\varphi_\eps)(x)\right)\in L^1(\R^n),
\]
 where $c_{n, s}$ is the constant of the fractional Laplacian, $\fls$,  and $K_{ab}(x, y)$ is the kernel of the operator $\L_{ab,x}$. In particular, 
\begin{equation}
\label{eq:fromexp}
\tilde \L^{(\eps)}_{ab, x} v (x) = v(x)\int_{B_{\eps/2}^c} \tilde K_\eps(x, y)\, dy - \int_{B_{\eps/2}^c(x)} v(z) \tilde K_\eps(x, z-x)\, dz.
\end{equation}
Let now $x\in B_{3/4}$ fixed, and let 
\begin{equation}
\label{eq:fromexp2}
\rho = \min\left\{\frac{\eps}{4}, \frac12 \dist(x, \partial B_{3/4})\right\} = \min\left\{\frac{\eps}{4},\frac12 \left(\frac34 - |x|\right)\right\}>0.
\end{equation}

Let us bound, for $h\in B_\rho$,
\[
|\tilde \L^{(\eps)}_{ab, x} v(x+h) - \tilde \L^{(\eps)}_{ab, x} v(x)|\le I + II + III, 
\]
where, for any $\mu \in (0, 1)$ we have 
\[
I = |v(x + h) - v(x)| \int_{B_{\eps/2}^c} \tilde K_\eps(x+h, y)\, dy\le C_\rho [v]_{C^\mu(B_\rho(x))}|h|^\mu,
\]
as well as (since $v$ is bounded)
\[
\begin{split}
II & = |v(x) | \int_{B_{\eps/2}^c}|\tilde K_\eps(x+h, y) - \tilde K_\eps(x, y)|\, dy\\
& \le |v(x)| \int_{B_{\eps/2}^c}|(K_{ab}(\cdot, y)*\varphi_\eps)(x+h) - (K_{ab}(\cdot, y)*\varphi_\eps)(x)|\, dy\\ 
& \le C_\rho |h|,
\end{split}
\]
(where we used that the regularized kernels $(K_{ab}(\cdot, y)*\varphi_\eps)(x)$ are uniformly Lipschitz in $(a, b)\in \A\times \B$, but not in $\eps$ as $\eps\downarrow 0$), and  
\[
\begin{split}
III& \le \int_{B^c_{\eps/4}(x)} |v(z)| |\tilde K_\eps(x, z-x) - \tilde K_\eps(x, z-x-h)|\, dz\\
& \le C_\rho \int_{B^c_{\eps/4}(x)}   |\tilde K_\eps(x, z-x) - \tilde K_\eps(x, z-x-h)|\, dz \le C_\rho |h|^\theta,
\end{split}
\]
 since $\I \in \II^\omega_s(\lambda, \Lambda; \theta)$.

Thanks to the previous bounds we have
\[
\begin{split}
[\tilde \L^{(\eps)}_{ab, x} v]_{C^\mu(B_{\rho}(x))} &
 \le C_\rho \left( \| v\|_{C^\mu(B_{\rho}(x))} +\|v\|_{L^\infty (\R^n)}\right),
\end{split}
\]
where $C_\eps$ is independent of $(a, b)\in \A\times\B$.  In \eqref{eq:Iepssplit} we can therefore bound the H\"older semi-norms of $f_\eps$ (being the $\inf\sup$ of H\"older functions) as
\[
[f_\eps]_{C^\mu(B_{\rho}(x))}\le C_\rho \left( \| v\|_{C^\mu(B_{\rho}(x))} +\|v\|_{L^\infty (\R^n)}+ C_\circ\right).
\]
Thus, we obtain that 
\begin{equation}
\label{eq:vimplication}
v\in C_{\rm loc}^\mu(B_{3/4})\quad\text{and}\quad 0\le \mu \le\theta \quad \Longrightarrow \quad f_\eps\in C_{\rm loc}^\mu(B_{3/4}),
\end{equation}
in a qualitative way.

Moreover, $v$ satisfies, by assumption
\[
\fls v = f_\eps\quad\text{in}\quad B_{3/4}. 
\]
Hence, we can now use interior estimates for viscosity solutions with  the fractional Laplacian, Proposition~\ref{prop:viscosity_fls} together with a bootstrap argument to conclude: to begin with, we already know that $v\in C^\gamma(B_{3/4})$, hence by \eqref{eq:vimplication} we have $f_\eps\in C^\gamma(B_{3/4})$ and by the interior estimates in Proposition~\ref{prop:viscosity_fls} $v\in C^{2s+\min\{\gamma, \theta\}}(B_{3/4})$. If $\theta > \gamma$, we can iteratively repeat this until $\gamma+ ms > \theta$ for some $m\in \N$, at which point we have to stop when we reach $C^\theta$ regularity of $f_\eps$. A final application of interior estimates implies $C^{2s+\theta}$  regularity of $v$. If $\theta = 0$, we only apply the iteration once. 
\end{proof}

We can finally prove the regularization result:
\begin{proof}[Proof of Proposition~\ref{prop:regularization}]\label{proof:thmregularization}
We construct $\I_\eps$ and $u^{(\eps)}$ as \eqref{eq:Iepsisoftheform} and \eqref{eq:uepsqualdef}. Then, Lemma~\ref{lem:weakconv} gives the weak convergence of $\I_\eps$ to  $\I$, and Lemma~\ref{lem:ueps1} and a covering argument give  the locally uniform convergence of $u^{(\eps)}$ in $B_{3/4}$ (by Arzel\`a-Ascoli, up to taking subsequences), towards some function $\tilde u  \in C(B_{3/4})\cap L^\infty(\R^n)$.    Hence we are in a situation where we can apply \cite[Lemma 4.3]{CS11} to deduce that $\tilde u\in C(B_{3/4})\cap   L^\infty(\R^n)$ satisfies
\[
\left\{
\begin{array}{rcll}
\I(\tilde u, x) & = & 0 & \quad\text{in}\quad B_{3/4}\\
\tilde u & = & u & \quad\text{in}\quad \R^n\setminus B_{3/4}.
\end{array}
\right.
\]
By the uniqueness of continuous viscosity solutions we have $\tilde u = u$, and moreover $u\in C(B_1)$. The  interior regularity is due to Lemma~\ref{lem:regularization}. This completes the proof. 
\end{proof}

\section{Approximation by strong solutions} \label{ssec:classicalapprox}

Proposition~\ref{prop:regularization}   gives an approximating sequence to a viscosity solution by \emph{smoother} solutions, which in the case $\theta > 0$ are strong (i.e., $C^{2s+}$).  Let us now very briefly show that, with a bit more of work, also in the  most general case $\theta = 0$ we can consider  strong solutions as the approximating sequence. We refer to \cite{Kri13} for a similar procedure  in the case $s > 1/2$.  We believe that part of the appeal of the following proof lies in its simplicity. 
%\footnote{An adaptation of \cite{CS11b, Kri13} was performed in \cite{Ser15} as well, where unfortunately Remark 4.1 and equation (4.8) are not correct, and hence the proof does not work.}
 
\begin{prop}
\label{prop:regularization_2}
Let $s\in (0, 1)$, and let $\I \in \II^\omega_s(\lambda, \Lambda)$. Let $u\in C(B_1)\cap L^\infty(\R^n)$ be any viscosity solution of
\[
\I (u, x) = 0\quad\text{in}\quad B_1. 
\]
Then, there exist $\delta > 0$, a sequence of functions,
\[  C^{2s+\delta}_{\rm loc}(B_{3/4}) \cap {C^\delta_c(\R^n)} \ni u^{(\eps)} \to u \  \text{locally uniformly in $B_{3/4}$   and  in $L^1(\R^n; w_s)$},\]
 and  a sequence of operators $\hat\I_\eps\in \II^\omega_s(\lambda, \Lambda)$ of the form \eqref{eq:Iepsisoftheform2},  such that
\[
\begin{split}
&\hat \I_\eps(u^{(\eps)}, x)   =  0  \quad\text{in } B_{3/4}\\ 
& \hat\I_\eps\rightharpoonup \I \quad\text{in the sense of Definition~\ref{defi:conv_I}},
\end{split}
\]
as $\eps\downarrow 0$. Moreover, we have 
\begin{equation}
\label{eq:bound_reg_2}
 \|u^{(\eps)}\|_{L^\infty(\R^n)}\le C \left(\| u\|_{L^\infty(\R^n)} + \|\I(0, x)\|_{L^\infty(B_{3/4})}+\omega(\eps)\right)
\end{equation}
for some $C$ depending only on $n$, $s$, $\lambda$, and $\Lambda$.
\end{prop}

In order to prove it, we proceed following a similar strategy to the one before. Now, however, we need to regularize the $c_{ab}(x)$ in the definition of~$\I$, as well as the value of $u$ outside of $B_{3/4}$. We will do that by means of a convolution.

We   define $\hat \I_\eps$ analogously to \eqref{eq:Iepsisoftheform} but also regularizing the terms $c_{ab}(x)$. That is, for any $\I$ of the form \eqref{eq:Iofheform00} we consider 
\begin{equation}
\label{eq:Iepsisoftheform2}
\hat \I_\eps (u, x) := \inf_{b\in \B}\sup_{a\in \A}\left\{-\L_{ab, x}^{(\eps)} u(x) + c^{(\eps)}_{ab}(x)\right\},\qquad \L_{ab, x}^{(\eps)}\in \LL^\omega_s(\lambda, \Lambda),
\end{equation}
where $\L_{ab, x}^{(\eps)}$ are the corresponding operators to $\L_{ab, x}$ but with kernel given by \eqref{eq:Kepsdef}, and where $c_{ab}^{(\eps)}(x) := (\varphi_\eps * c_{ab})(x)$ (recall \eqref{eq:mollifiervarphi}-\eqref{eq:mollifiervarphi2}).  Lemma~\ref{lem:weakconv} still holds in this case:
 \begin{lem} 
 \label{lem:weakconv2}
Let $s\in (0, 1)$, and let $\I, \hat \I_\eps\in \II^\omega_s(\lambda, \Lambda)$ be as above. Then 
 \[
\hat \I_\eps\rightharpoonup \I\quad\text{in}\quad \R^n,\quad \text{as}\quad \eps\downarrow 0,
 \]
 in the sense of Definition~\ref{defi:conv_I}. 
 \end{lem}
 \begin{proof}
The proof is exactly the same as that of Lemma~\ref{lem:weakconv}, where we now use that since $c_{ab}(x)$ are equicontinuous, then $c_{ab}^{(\eps)}(x)$ converges locally uniformly  to $c_{ab}(x)$ as $\eps\downarrow 0$ independently of $(a, b)\in \A\times \B$ (that is, depending only on~$\omega$). 
 \end{proof}
 
 If $\I \in \II^\omega_s(\lambda, \Lambda)$ and $u\in C(B_1)\cap L^\infty(\R^n)$ is a viscosity solution to
\begin{equation}
\label{eq:estrelletax22}
\I (u, x) = 0\quad\text{in}\quad B_1,
\end{equation}
we   define our new functions $u^{(\eps)}$ to be the unique solution to (given, for example, again by \cite{Ser15, Mou19}) 
\begin{equation}
\label{eq:uepsqualdef2}
\left\{
\begin{array}{rcll}
\hat \I_\eps(u^{(\eps)}, x) & = & 0 & \quad \text{in}\quad B_{3/4}\\
u^{(\eps)} & = & (u\chi_{B_{1/\eps}}) * \varphi_\eps & \quad \text{in}\quad \R^n\setminus B_{3/4}. 
\end{array}
\right.
\end{equation}

In doing so, the following analogue of 	Lemma~\ref{lem:ueps1} also holds now:

\begin{lem}
\label{lem:ueps12}
Let $s\in (0, 1)$ and $\I \in \III$. Let $u\in C(B_1)\cap L^\infty(\R^n)$ be any viscosity solution of \eqref{eq:estrelletax22}, and let $u^{(\eps)}\in  C(B_{3/4})\cap L^\infty(\R^n)$ be the unique solution of \eqref{eq:uepsqualdef2}. Let  $\gamma > 0$ be given by Theorem~\ref{C^alpha-bmc}. Then 
\[
\| u^{(\eps)}\|_{L^\infty(\R^n)}+\|u^{(\eps)}\|_{C^\gamma(B_{1/2})}\le C \left(\| u\|_{L^\infty(\R^n)} + \|\I(0, x)\|_{L^\infty(B_{3/4})}+\omega(\eps)\right),
\]
 for some $C$ depending only on $n$, $s$, $\lambda$, and $\Lambda$.
\end{lem} 
\begin{proof}
The proof is the same as that of Lemma~\ref{lem:ueps1}, by using that 
\[
\|(u\chi_{B_{1/\eps}}) * \varphi_\eps\|_{L^\infty(\R^n)}\le  \|u \chi_{B_{1/\eps}}\|_{L^\infty(\R^n)}\le  \|u\|_{L^\infty(\R^n)}. 
\]
The main difference is the appearance of $\omega(\eps)$ on the right-hand side of the estimate. This is because  we now have $\|c_{ab}^{(\eps)} - c_{ab}\|_{L^\infty(B_{3/4})}\le \omega(\eps)$, and so $\|\hat{\I}_\eps(0, x)-\I(0, x)\|_{L^\infty(B_{3/4})}\le \omega(\eps)$.
\end{proof}

By regularizing the boundary datum we  have now improved the regularity of $u^{(\eps)}$ with respect to the previous case, Lemma~\ref{lem:regularization}: 

\begin{lem}
\label{lem:regularization2}
Let $s\in (0, 1)$. Let $u^{(\eps)}$ be defined by \eqref{eq:uepsqualdef2}, for a fixed $\I \in \II^\omega_s(\lambda, \Lambda)$.  Then, there exists $\delta  >0$ independent of $\eps > 0$ such that $u^{(\eps)}\in C^{2s+\delta}_{\rm loc}(B_{3/4})\cap C_c^\delta(\R^n)$.
\end{lem}
\begin{proof}
For the sake of readability, we denote $v = u^{(\eps)}$. Observe that  the exterior datum satisfies
\[
\|\nabla ((u\chi_{B_{1/\eps}}) * \varphi_\eps)\|_{L^\infty(\R^n)}\le C_\eps,
\]
for some $C_\eps$ that might blow-up as $\eps\downarrow 0$. This is enough to deduce that, from the boundary regularity in Lemma~\ref{lem:bdry_reg_int}, there exists some $\delta > 0$ (independent of $\eps> 0$) such that $v \in C^\delta_c(\R^n)$.

As in Lemma~\ref{lem:regularization}, we rewrite the operator $\hat \I_\eps$ as 
\[
\hat\I_\eps(v, x)  =  -c^{-1}_{n,s}\fls v(x) + f_\eps(x),
\]
where 
\[
f_\eps(x) := \inf_{b\in \B}\sup_{a\in \A}\left\{\tilde \L^{(\eps)}_{ab, x} v (x) + c^{(\eps)}_{ab}(x)\right\}
\]
and \eqref{eq:fromexp} holds.

For $x\in B_{3/4}$ fixed, we proceed as in Lemma~\ref{lem:regularization} taking $\rho$ as \eqref{eq:fromexp2} and bounding, for $h\in B_\rho$,
\[
|\tilde \L^{(\eps)}_{ab, x} v(x+h) - \tilde \L^{(\eps)}_{ab, x} v(x)|\le I + II + \overline{III}, 
\]
where, for any $\mu \in (0, 1)$ we have 
\[
I \le C_\rho [v]_{C^\mu(B_\rho(x))}|h|^\mu,\qquad  II \le  C_\rho |h|,
\]
and now we rewrite $III$ as 
\[
\overline{III}\le  \overline{III}_i + \overline{III}_{ii}
\]
with 
\[
\overline{III}_i = \int_{B_{\eps/2}^c} |v(x+y+h) - v(x+h)|\tilde K_{\eps}(x+h, y)\, dy\le C_\rho |h|^\delta [v]_{C^\delta(\R^n)},
\]
and 
\[
\overline{III}_{ii}  = \int_{B_{\eps/2}^c} \hspace{-2mm} |v(x+y)| \left| (K_{ab}(\cdot, y)*\varphi_\eps)(x+h) - (K_{ab}(\cdot, y)*\varphi_\eps)(x)\right|\, dy \le C_\rho |h|
\]
(proceeding as in the bound of $II$).

Together with the fact that $c_{ab}^{(\eps)}\in C^\infty$, we get that
\[
f_\eps(x) \in C^\delta_{\rm loc}(B_{3/4}) 
\]
for some $\delta > 0$ independent of $\eps$. By the interior estimates for viscosity solutions with the fractional Laplacian, Proposition~\ref{prop:viscosity_fls}, we deduce $v\in C^{2s+\delta}_{\rm loc}(B_{3/4})$, as wanted. 
\end{proof}

We can finally prove Proposition~\ref{prop:regularization_2}:
\begin{proof}[Proof of Proposition~\ref{prop:regularization_2}]
We proceed as in the proof of Proposition~\ref{prop:regularization}, with the corresponding changes in this new situation. 

We construct $\hat \I_\eps$ and $u^{(\eps)}$ as \eqref{eq:Iepsisoftheform2} and \eqref{eq:uepsqualdef2}, and  Lemma~\ref{lem:weakconv2} gives the weak convergence of $\hat \I_\eps$ to  $\I$, while Lemma~\ref{lem:ueps12} and a covering argument give  the locally uniform convergence  in $B_{3/4}$ and the convergence  in $L^1(\R^n; w_s)$ of $u^{(\eps)}$ to some $\tilde u  \in C(B_{3/4})\cap L^\infty(\R^n)$.  The stability of viscosity solutions under limits (see \cite[Lemma 4.3]{CS11}) implies that $\tilde u$ satisfies
\[
\left\{
\begin{array}{rcll}
\I(\tilde u, x) & = & 0 & \quad\text{in}\quad B_{3/4}\\
\tilde u & = & u & \quad\text{in}\quad \R^n\setminus B_{3/4},
\end{array}
\right.
\]
and by uniqueness, we have $\tilde u = u$, and $u\in C(B_1)$. The  qualitative interior regularity is due to Lemma~\ref{lem:regularization2} and  this completes the proof. 
\end{proof}

\subsection{Equivalence between viscosity and distributional solutions}\label{ssec:equivalence} As a consequence of Proposition~\ref{prop:regularization_2} we obtain that, in the \emph{linear} and \emph{translation invariant} case (taking operators $\L\in \LLL$, whose kernel does not depend on $x$), the notions of viscosity  and distributional   solution are equivalent:

\begin{lem}
\label{lem:visc_dist}
Let $s\in (0, 1)$, $u\in L^\infty(\R^n)$,  $f \in C(B_1)$, and $\L$ be a translation invariant operator with kernel comparable to the fractional Laplacian: 
\[
\L u(x)  = {\rm P.V.} \int_{\R^n}\big(u(x)-u(x+y)\big)K(y) \, dy 
\]
with
\[
K( y) = K(  -y)\ \quad\text{and} \quad 0< \lambda \le |y|^{n+2s} K( y) \le\Lambda\qquad\text{in}\quad   \R^n.
\]

Then, $u$ solves $\L u = f$ in $B_1$ in the distributional sense if and only if it does so in the  viscosity sense.
\end{lem}
\begin{proof}
If $u$ is a distributional solution, it is continuous (by \cite[Theorem 3.8]{DRSV22}), and we can regularize it and consider  (recall \eqref{eq:mollifiervarphi}-\eqref{eq:mollifiervarphi2})
\[
u_\eps := u * \varphi_\eps, 
\]
for some smooth mollifier $\varphi_\eps = \eps^{-n}\varphi(x/\eps)$. Then $u_\eps$ satisfies 
\[
\L u_\eps = f_\eps \quad\text{in}\quad B_{1-\eps}
\]
in the strong sense, and therefore, in the viscosity sense as well. Taking the limit $\eps \downarrow 0$, by \cite[Lemma 4.3]{CS11}  $u$ is a viscosity solution to $\L u = f$ in $B_1$.

Conversely, if $u\in C(B_1)$ is a viscosity solution to the equation, by Proposition~\ref{prop:regularization_2} it can be approximated by  strong  solutions (and therefore, distributional solutions) $u_\eps\to u$  to an equation of the form 
\[
\hat{\L}_\eps u_\eps = f_\eps \quad\text{in}\quad {B_{3/4}},
\]
with a sequence of explicit operators $\tilde \L_\eps$. 

Then, the limit $\eps\downarrow 0$ is a distributional solution (see \cite[Lemma 3.2 and proof of Theorem 3.8]{DRSV22}) to $\L_\infty u = f$ in $B_{3/4}$, where by construction $\L_\infty = \L$. A covering argument, yields that $\L u = f $ in $B_1$ in the distributional sense. 
\end{proof}
 \begin{rem}
 Lemma~\ref{lem:visc_dist} may also apply to non-translation invariant kernels, as long as they admit both definition of viscosity and distributional solutions (which requires regularity in $x$). 	
 \end{rem}

\section{Proof of main result}
\label{sec:proofmain}
The goal of this section is to finally prove that we can actually approximate viscosity solutions by $C^\infty_c(\R^n)$ solutions, Theorem~\ref{thm:cinftysol00}.  

%\begin{thm}
%\label{thm:cinftysol}
%Let $s\in (0, 1)$, and let $\I \in \III$ of the form \eqref{eq:ILab} {\color{red}with modulus $\sigma$***}. Let $u\in C(B_1)\cap L^1_{\omega_s}(\R^n)$ be any viscosity solution of
%\[
%\I (u, x) = 0\quad\text{in}\quad B_1. 
%\]
%
%Then, there exist  a sequence of functions $u^{(\eps)}\in C^\infty_c(\R^n)$, $f_\eps\in C^\infty(\R^n)$, and a sequence of operators $\I_\eps\in \III$  of the form \eqref{eq:ILab} {\color{red}with modulus $\sigma$***} such that, 
%\[
%\I_\eps(u_\eps, x) = f_\eps(x) \quad\text{in}\quad B_{3/4}
%\]
%and 
%\[
%\begin{array}{rcll}
% f_\eps(x) &\to& 0&\quad\text{uniformly in $B_{3/4}$,}\\
%\I_\eps(0, x) &\to& \I(0, x)&\quad\text{uniformly in $B_{3/4}$,}\\
%u_\eps & \to & u&\quad\text{uniformly in $B_{3/4}$,}\\
%u_\eps & \to & u&\quad\text{in $L^1_{\omega_s}(\R^n)$,}\\
%\end{array}
%\]
%as $\eps\downarrow 0$. 
%\end{thm}

In order do it, we will combine the approximation by strong solutions in Proposition~\ref{prop:regularization_2} with the next result, in which we provide a way to regularize the operator $\I$ itself.

\begin{prop}
\label{prop:cinfty1}
Let $s\in (0, 1)$, and let $\I\in  \II_s^{\omega}(\lambda, \Lambda)\cap \II_s^{\infty}(\lambda, \Lambda)$. Let $u\in C^{2s+\delta}_{\rm loc}(B_1)\cap C^\delta_c(\R^n)$ be any   solution of 
\[
\I(u, x) = f(x) \quad\text{in}\quad B_1
\] 
for some $f\in C(B_1)$ and $\delta > 0$.  Let $(\varphi_\eps)_{\eps > 0}$ be given by  \eqref{eq:mollifiervarphi}-\eqref{eq:mollifiervarphi2}.

Then, there exist $\I_\eps\in \II_s^{\omega}(\lambda, \Lambda)\cap \II_s^{\infty}(\lambda, \Lambda)$  such that the sequence $u_\eps := u * \varphi_\eps\in C^\infty_c(\R^n)$ satisfies 
\[
\I_\eps(u_\eps, x) = f_\eps(x)\quad\text{in}\quad B_1
\]
for some $f_\eps\in C^\infty(B_1)$ such that 
\[
f_\eps \to  f\quad\text{uniformly in $B_{3/4}$  as $\eps\downarrow 0$.}
\]
Moreover, 
\[
\I_\eps(0, x) \to  \I(0, x) \quad\text{uniformly in $B_{3/4}$  as $\eps\downarrow 0$.}
\]

\end{prop}

\begin{proof}
We divide the proof into four steps.  For the sake of readability, we assume $f = 0$. The general case follows analogously by taking $\I(\cdot, x) - f(x)$.
\begin{steps}
\item We define $c_{ab}^{(\eps)} := c_{ab}*\varphi_\eps\in C^\infty(\R^n)$ and we consider 
\[
\hat{\I}_\eps (v, x) := \inf_{b\in \B}\sup_{a\in \A}\left\{-\L_{ab, x}  v(x) + c^{(\eps)}_{ab}(x)\right\},\qquad \L_{ab, x}\in \LL^{\omega}_s(\lambda, \Lambda)\cap \LL^{ \infty}_s(\lambda, \Lambda).
\]
Notice that $\L_{ab, x} u_\eps \in C^\delta_{\rm loc}(B_1)$ (see Lemma~\ref{lem:Lu_2}) with local uniform (in $a$, $b$, and $\eps$) estimates in $B_1$, as well as $\L_{ab, x} u_\eps \in C^\infty(\R^n)$ (locally uniformly in $a$ and $b$, but not in $\eps$) with vanishing derivatives at infinity. Since $c_{ab}$ are equicontinuous, the family $\L_{ab, x}  u_\eps(x) + c^{(\eps)}_{ab}(x)$ is locally equicontinuous in $B_1$. In particular, there exists a modulus of continuity $\omega_\circ$   such that $\L_{ab, x}  u_\eps(x) + c^{(\eps)}_{ab}(x)$ is continuous with modulus $\omega_\circ$ in $B_{3/4}$, for all $(a, b)\in \A\times \B$ and $\eps \ge 0$.  

Hence, in fact, $(\hat{\I}_\eps(u_\eps, x))_{\eps\ge 0}$ is locally equicontinuous in $B_1$, and 
\begin{equation}
\label{eq:hatieps}
  \hat{\I}_\eps(u_\eps, x)\to 0\quad\text{locally uniformly in $B_1$},
\end{equation}
(recall $f \equiv 0$) as well as 
\begin{equation}
\label{eq:hatieps2}
  \hat{\I}_\eps(0, x)\to \I(0, x) \quad\text{locally uniformly in $B_1$}.
\end{equation}

\item We now consider, for any $\eps > 0$ fixed, a finite collection of points $G_\eps :=\{y_1,\dots,y_{N_\eps}\}$ with $y_i \in B_{3/4}$ for $1\le i\le N_\eps$ such that $\dist(z, G_\eps) \le \zeta$ for all $z\in B_{3/4}$, where $\zeta = \zeta(\eps)$ is chosen small enough so that $\omega_\circ(\zeta)\le \eps/4$  (where $\omega_\circ$ is the modulus of continuity of the previous step).

We  want to take a finite redefinition of $\hat{\I}_\eps$ such that its value at $u_\eps$ and $0$ is not altered too much. Namely, for any $y_i\in G_\eps$, we consider $b_i, b_{N_\eps+i}\in \B$   such that if we define 
\[
\begin{split}
\mathcal{G}_i(v, x) & := \sup_{a\in \A}\left\{-\L_{ab_i, x}  v(x) + c^{(\eps)}_{ab_i}(x)\right\}\quad\text{for}\quad 1 \le i \le 2N_\eps
\end{split}
\]
then 
\[
\begin{array}{rcccll}
 0&\le &\mathcal{G}_i(u_\eps, y_i) - \hat{\I}_\eps(u_\eps, y_i)  & \le& {\eps}/{4},&\\[0.1cm]
 0&\le  &\mathcal{G}_{N_\eps +i}(0, y_i) - \hat{\I}_\eps(0, y_i)  & \le&  {\eps}/{4} & \quad \text{for} \quad 1\le i \le N_\eps.
\end{array}
\]
 Together with the fact that $\mathcal{G}_i(v, x) \ge \hat{\I}_\eps(v, x)$ in $\R^n$ for all $1 \le i \le 2N_\eps$, and from the choice of $\zeta$, we have 
\begin{equation}
\label{eq:tocombg_i}
\begin{array}{rcccll}
  0&\le&{\displaystyle \inf_{1\le i \le 2N_\eps}} \mathcal{G}_i(u_\eps, x) - \hat{\I}_\eps (u_\eps, x)& \le  {\eps}/{2}&\quad\text{in}\quad B_{3/4},\\[0.2cm]
 0&\le & {\displaystyle \inf_{1\le i \le 2N_\eps} }\mathcal{G}_i(0,  x) - \hat{\I}_\eps (0, x)& \le  {\eps}/{2}&\quad\text{in}\quad B_{3/4}. 
\end{array}
\end{equation}
Similarly, for each $1\le i \le 2N_\eps$ fixed, and for any $y_j\in G_\eps$ we consider $a_{ij}, a_{i, N_\eps+j}\in \A$ such that 
\[
\begin{split}
\left|-\L_{a_{ij} b_i, y_j} u_\eps(y_j) + c_{a_{i,j} b_i}^{(\eps)}(y_j) - \mathcal{G}_i(u_\eps, y_j)\right|& \le  {\eps}/{4},\\
\left|c_{a_{i,N_\eps + j} b_i}^{(\eps)}(y_j) - \mathcal{G}_i(0, y_j)\right|& \le  {\eps}/{4}\quad\text{for}\quad 1 \le j \le N_\eps.
\end{split}
\]
In particular, again by the choice of $\zeta$ above, we have that 
\[
\begin{split}
\left|\sup_{1\le j \le 2N_\eps} \left\{-\L_{a_{ij} b_i, x} u_\eps(x) + c_{a_{ij} b_i}^{(\eps)}(x)\right\} - \mathcal{G}_i(u_\eps, x)\right|& \le {\eps}/{2}\quad\text{in}\quad B_{3/4},\\
\left|\sup_{1\le j \le 2N_\eps}   c_{a_{ij} b_i}^{(\eps)}(x) - \mathcal{G}_i(0, x)\right|& \le  {\eps}/{2}\quad\text{in}\quad B_{3/4}.
\end{split}
\]
Combined with \eqref{eq:tocombg_i} we get 
\[
\begin{split}
\left| \inf_{1\le i \le 2N_\eps} \sup_{1\le j \le 2N_\eps} \left\{-\L_{a_{ij} b_i, x} u_\eps(x) + c_{a_{ij} b_i}^{(\eps)}(x)\right\} - \hat{\I}_\eps (u_\eps, x)\right|& \le {\eps}\quad\text{in}\quad B_{3/4},\\
\left| \inf_{1\le i \le 2N_\eps} \sup_{1\le j \le 2N_\eps} \left\{ c_{a_{ij} b_i}^{(\eps)}(x)\right\} - \hat{\I}_\eps (0, x)\right|& \le {\eps}\quad\text{in}\quad B_{3/4}.
\end{split}
\]
Thus, we can define 
\[
\I_\eps^*(v, x ) := \inf_{1\le i \le 2N_\eps} \sup_{1\le j \le2 N_\eps} \left\{-\tilde \L_{ij, x} v(x) + \tilde c_{ij}^{(\eps)}(x)\right\}
\]
where
\[
\tilde\L_{ij, x} := \L_{a_{ij} b_i, x}\in   \LL^\omega_s(\lambda, \Lambda)\cap \LL^\infty_s(\lambda, \Lambda)\quad\text{and}\quad \tilde c_{ij}^{(\eps)}= c_{a_{ij} b_i}^{(\eps)}\quad\text{for}\quad 1\le i,j\le 2N_\eps
\]
and we have that 
\begin{equation}
\label{eq:Iepfftohat}
\begin{split}
\big|\I_\eps^*(u_\eps, x) -\hat{\I}_\eps(u_\eps, x)\big|&\le \eps\quad\text{in}\quad B_{3/4},\\
\big|\I_\eps^*(0, x) -\hat{\I}_\eps(0, x)\big|&\le \eps\quad\text{in}\quad B_{3/4}.
\end{split}
\end{equation}
The key  difference now is that $\I^*_\eps$ is a \emph{finite} $\inf\sup$.

\item Let us denote, for the sake of readability, $N := 2N_\eps$. We define $F_\eps:\R^{N\times N} \to \R$ as 
\[
F_\eps(\{x_{ij}\}_{1\le i,j\le N}) = F_\eps 
\begin{pmatrix}
x_{11} & x_{12} & \dots & x_{1 N}\\
x_{21} & x_{22} & \dots & x_{2 N}\\
\vdots & \vdots & \ddots & \vdots\\
x_{N 1} & x_{N 2} & \dots & x_{N N}
\end{pmatrix} = \inf_{1\le i \le N} \sup_{1 \le j \le N} x_{ij},
\]
so that 
\begin{equation}
\label{eq:Iepsf}
\I_\eps^*(v, x) = F_\eps\left(\left\{-\tilde \L_{ij, x} v(x) + \tilde c_{ij}^{(\eps)}(x)\right\}_{1 \le i,j\le N}\right).
\end{equation}

Then, $F_\eps$ is a piecewise linear function with $|\nabla F_\eps| = 1$ a.e. and such that for a.e.  $x\in \R^{N\times N}$, $\nabla F_\eps(x) \in  \{\be_{ij}\}_{1\le i,j\le N}$, where $\be_{ij} \in \R^{N\times N}$ is the matrix with $(\be_{ij})_{ij} = 1$ and $(\be_{ij})_{k\ell} = 0$ for all $(k, \ell)\neq (i, j)$. 

In particular, by considering a regularization $F_\eps^r := F_\eps * \varphi_\eps$, where $\varphi_\eps\in C^\infty_c(B_\eps)$ with $B_\eps\in \R^{N\times N}$ (see \eqref{eq:mollifiervarphi}-\eqref{eq:mollifiervarphi2}) we have that $F_\eps^r \in C^\infty(\R^{N\times N})$ with 
\[
{\rm Grad}( F_\eps^r) := \bigcup_{x\in \R^{N\times N}} \nabla F_\eps^r(x)\subset \partial \, {\rm Conv}\big( \{\be_{ij}\}_{1\le i,j\le N}\big),
\]
where ${\rm Conv}(A)$ denotes the convex hull of $A\in \R^{N\times N}$. Since $|\nabla F_\eps |\le 1$, 
\begin{equation}
\label{eq:FepsFepsr}
\|F_\eps - F_\eps^r\|_{L^\infty(\R^{N\times N})}\le \eps,
\end{equation}
and   we can write it as 
\[
F_\eps^r(x) = \inf_{z\in \R^{N\times N}}\sup_{M\in {\rm Grad}(F_\eps^r)} \big\{M \cdot x - M\cdot z+F^r_\eps(z)\big\}.
\]
(This representation formula is valid for any Lipschitz function.) We then define
\begin{equation}
\label{eq:Iepsvfollows}
\I_\eps(v, x) := F_\eps^r\left(\left\{-\tilde \L_{ij, x} v(x) + \tilde c_{ij}^{(\eps)}(x)\right\}_{1 \le i,j\le N}\right),
\end{equation}
so that\footnote{   If $f \not\equiv 0$, we would have now $\I_\eps(v, x)  - (f*\varphi_\eps)(x)$ as a regularized version of $\I(v, x) - f(x)$, since $\sum_{i, j} M_{ij} = 1$.}
\begin{equation}
\label{eq:repIeps}
\begin{split}
\I_\eps(v, x) & = \inf_{z\in \R^{N\times N}}\sup_{M\in {\rm Grad}(F_\eps^r)} \left\{ \sum_{i,j =1}^{N} \left(-M_{ij} \tilde \L_{ij, x} v(x) + M_{ij} \tilde c_{ij}^{(\eps)}(x)\right) +C_{M,z}^\eps\right\}\\
& \hspace{-1.2cm}= \inf_{z\in \R^{N\times N}}\sup_{M\in {\rm Grad}(F_\eps^r)} \left\{ - \left(\sum_{i,j =1}^{N} M_{ij} \tilde \L_{ij, x} \right) v(x) + \left( \sum_{i, j = 1}^N M_{ij} \tilde c_{ij}^{(\eps)}(x) +C_{M,z}^\eps\right) \right\},
\end{split}
\end{equation}
where 
\[
C_{M, z}^\eps := F_\eps^r(z) - M\cdot z.
\]
In particular, since $\sum_{i,j=1}^{N} M_{ij} = 1$, $M_{ij}\ge 0$, and $\LLL\cap \LL^{ \infty}_s(\lambda, \Lambda)$ is convex, we have that $\I_\eps\in \III\cap \II^{  \infty}_s(\lambda, \Lambda)$ with
\[
\I  (v, x) = \inf_{b\in \hat  \B}\sup_{a\in \hat  \A}\left\{-\hat  \L^{(\eps)}_{ab, x}  v(x) + \hat c^{(\eps)}_{ab}(x)\right\},\qquad \hat  \L^{(\eps)}_{ab} \in \LL^\omega_s(\lambda, \Lambda)\cap \LL^\infty_s(\lambda, \Lambda),
\]  
and where $\hat  c^{(\eps)}_{ab}$ are equicontinuous with modulus $\omega$  (the same as for $c_{ab}$). 

\item To finish, we notice that by the chain rule, since $\tilde \L_{ij, x} u_\eps, \tilde c_{ij}^{(\eps)}\in C^\infty(\R^n)$, it follows from \eqref{eq:Iepsvfollows} that $\I_\eps(u_\eps, x) \in C^\infty(\R^n)$. 

Moreover, thanks to \eqref{eq:FepsFepsr}-\eqref{eq:Iepsf} together with \eqref{eq:Iepfftohat} and \eqref{eq:hatieps}-\eqref{eq:hatieps2}, we have 
\[
\begin{array}{ll}
\I_\eps(u_\eps, x)  \to 0&\quad\text{uniformly in $B_{3/4}$}\\
\I_\eps(0, x)  \to \I(0, x) &\quad\text{uniformly in $B_{3/4}$}.
\end{array}
\]
This completes the proof. \qedhere
\end{steps}
\end{proof}

With this, we can complete the approximation result by $C^\infty_c$ solutions: 

\begin{proof}[Proof of Theorem~\ref{thm:cinftysol00}]
By defining the operator $\J(\cdot, x) := \I(\cdot, x) - f(x)$, we consider first the sequence of functions $u^{(\eps)}$ from Proposition~\ref{prop:regularization_2} applied with operator $\J$ in $B_{5/6}$ (after a scaling argument), so $u^{(\eps)}\in C^{2s+\delta}(B_{5/6})\cap C^\delta_c(\R^n)$. Notice that this also generates a sequence of operators $\hat \J_\eps(\cdot, x) = \hat \I_\eps(\cdot, x) - (f*\varphi_\eps)(x)$. Observe, also, that $\mathcal{J} \in \II_s^\infty(\lambda, \Lambda)$ as well (see Remark~\ref{rem:reg_in_x}).

Each $u^{(\eps)}$ can then be regularized by applying  Proposition~\ref{prop:cinfty1} (rescaled to $B_{5/6}$), which together with a diagonal argument yields the desired result. The bound on $\|u_\eps\|_{L^\infty(\R^n)}$ is a consequence of Lemma~\ref{lem:ueps12}. 
\end{proof}

\begin{rem}
\label{rem:reg_inherited}
In  Theorem~\ref{thm:cinftysol00} we have that, in fact, $f_\eps = f*\varphi_\eps$. Furthermore, notice that from the proof of  Proposition~\ref{prop:cinfty1}, and more precisely, from the representation \eqref{eq:repIeps} together with Lemma~\ref{lem:Lepsalpha}, we have that if $\I\in \II^\omega_s(\lambda, \Lambda; \theta)$ for some $\theta>0$, then $\I_\eps\in \II^\omega_s(\lambda, \Lambda; \theta)$ as well, with $[\I_\eps]^y_{\theta}\le C[\I]^y_{\theta}$, and $C$ depending only on $n$, $s$, $\lambda$, $\Lambda$, and $\theta$ (the regularity in $x$ is also preserved, since it is regularized with a convolution). Finally, also from \eqref{eq:repIeps}, if $\I$ is of the form \eqref{eq:Iofheform00}, and $\I_\eps$ is of the form 
\[
\I_\eps  (u, x) = \inf_{b'\in \B_\eps}\sup_{a'\in \A_\eps}\big\{-\L_{a'b', x}^{(\eps)}  u(x) + c_{a'b'}^{(\eps)}(x)\big\},\qquad \L_{a'b', x} \in \LL^\omega_s(\lambda, \Lambda),
\]
then for any $(a', b') \in \A_\eps\times \B_\eps$, 
\[
[c^{(\eps)}_{a'b'}]_{C^\mu(\R^n)}\le  \sup_{(a, b)\in \A\times \B} [c_{ab}]_{C^\mu(\R^n)},
\]
for $\mu > 0$. 

The same conclusion also holds for pointwise norms, like the ones in Remark~\ref{rem:pointwise} (thanks to \ref{rem:pointwise2}).
\end{rem}

\end{document}